\documentclass[11pt,reqno]{amsart}
\usepackage{amsmath, amssymb, amsthm, esint, verbatim, hyperref}
\usepackage{times}
\usepackage{dsfont}

\usepackage{wrapfig}
\usepackage{tikz}
\usetikzlibrary{decorations.fractals}

\numberwithin{equation}{section}

\hypersetup{
  pdftitle={Rectifiable Reifenberg and uniform positivity under almost calibrations},
  pdfauthor={Nicholas Edelen and Aaron Naber and Daniele Valtorta},
  pdfsubject={},
  pdfkeywords={},
  pdfpagelayout=SinglePage,
  pdfpagemode=UseOutlines,
  colorlinks,
  bookmarksopen,
  linkcolor=[rgb]{0,0,0.7},
  urlcolor=[rgb]{0,0,0.4},
  citecolor=[rgb]{0.4,0.1,0},
}

\newtheorem{theorem}{Theorem}[section]
\newtheorem{ctheorem}[theorem]{Conjectural Theorem}
\newtheorem{proposition}[theorem]{Proposition}
\newtheorem{observation}[theorem]{Observation}
\newtheorem{prop}[theorem]{Proposition}
\newtheorem{lemma}[theorem]{Lemma}
\newtheorem{corollary}[theorem]{Corollary}
\theoremstyle{definition}
\newtheorem{definition}[theorem]{Definition}
\theoremstyle{remark}
\newtheorem{remark}[theorem]{Remark}
\theoremstyle{remark}
\newtheorem{example}[theorem]{Example}
\theoremstyle{remark}
\newtheorem{note}[theorem]{Note}
\theoremstyle{remark}
\newtheorem{question}[theorem]{Question}
\theoremstyle{remark}
\newtheorem{conjecture}[theorem]{Conjecture}

\newcommand{\spt}{\mathrm{spt}}
\newcommand{\haus}{\mathcal{H}}

\newcommand{\cH}{\mathcal{H}}

\newcommand{\Lip}{\mathrm{Lip}}

\newcommand{\eps}{\epsilon}

\newcommand{\N}{\mathbb{N}}

\newcommand{\R}{\mathds{R}}
\newcommand{\C}{\mathds{C}}

\newcommand{\dC}{\mathds{C}}

\newcommand{\dR}{\mathds{R}}

\newcommand{\ton}[1]{\left(#1\right)}

\newcommand{\cur}[1]{\left\{#1\right\}}

\newcommand{\B}[2]{B_{#1}\ton{#2}}

\newcommand{\hol}{H\"older }

\title{Rectifiable Reifenberg and Uniform Positivity under Almost Calibrations}
\author{Nick Edelen, Aaron Naber and Daniele Valtorta}\thanks{N.E. was supported by NSF grant DMS-2204301, }
\date{\today}


\address{University of Notre Dame (USA)}
\email{nedelen@nd.edu}
\address{Northwestern University (USA)}
\email{anaber@math.northwestern.edu}
\address{University of Milano-Bicocca (EU)}
\email{daniele.valtorta@unimib.it}

\begin{document}
\begin{abstract}

The Reifenberg theorem \cite{reif_orig} tells us that if a set $S\subseteq B_2\subseteq \dR^n$ is uniformly close on all points and scales to a $k$-dimensional subspace, then $S$ is H\"older homeomorphic to a $k$-dimensional Euclidean ball.  In general this is sharp, for instance such an $S$ may have infinite volume, be fractal in nature, and have no rectifiable structure. \\

The goal of this note is to show that we can improve upon this for an almost calibrated Reifenberg set, or more generally under a positivity condition in the context of an $\epsilon$-calibration $\Omega$ .  An $\epsilon$-calibration is very general, the condition holds locally for all continuous $k$-forms such that $\Omega[L]\leq 1+\epsilon$ for all $k$-planes $L$. We say an oriented $k$-plane $L$ is $\alpha$-positive with respect to $\Omega$ if $\Omega[L]>\alpha>0$.  If $\Omega[L]>\alpha> 1-\epsilon$ then we call $L$ an $\epsilon$-calibrated plane.\\

The main result of this paper is then the following.  Assume at all points and scales $B_r(x)\subseteq B_2$ that $S$ is $\delta$-Hausdorff close to a subspace $L_{x,r}$ which is uniformly positive $\Omega[L_{x,r}]>\alpha $ with respect to an $\epsilon$-calibration.  Then $S$ is $k$-rectifiable with uniform volume bounds.

\end{abstract}
\maketitle
\tableofcontents

\section{Introduction}

For a closed set $S\subset \B 2 0\subseteq  \R^m$ and $B_r(x)\subseteq B_2(0)$, following \cite{toro:reifenberg} let us define $\theta(x,r)$ by 
\begin{align}
	\theta(k;x,r)=\theta(x,r)\equiv r^{-1}\inf_{L^k} d_H\big(S\cap B_r(x),{L}\cap B_r(x)\big)\, ,
\end{align}
where the inf is taken over all affine $k$-planes $L$ and $d_H$ is the Hausdorff distance. The classical Reifenberg theorem states that if $\theta(x,r)<\epsilon(n)$ for all $x\in S$ and $B_r(x)\subseteq B_2(0)$, then  $S\cap \B 1 0 $ is bi-H\"older equivalent to a flat disk $\B 1 0 \subset \R^k$.  We sometimes say in this case that $S$ satisfies the $\epsilon$-Reifenberg condition. This bi-\hol equivalence is proved by constructing smooth approximations of $S_r$ at all scales $r\in [0,1]$, and carefully analyzing their behavior, see Section \ref{ss:Sr_construction}.\\

As is well understood, being $\epsilon$-Reifenberg is not enough to guarantee $k$-dimensional volume bounds on $S$ or its rectifiability. Many papers have investigated how to obtain these extra results by adding additional constraints on the set $S$.  For instance in \cite{toro:reifenberg} it is shown that if one has the $L^2$ Dini estimate $\int \theta(x,r)^2 \frac{dr}{r}<\epsilon$ for all $x\in S$, then $S$ is bilipschitz to a Euclidean ball.  In \cite{NV_harmonic} it is shown that if one replaces the pointwise Dini estimate with an average $\int_S \int \theta(x,r)^2 \frac{dr}{r}<\epsilon$, then $S$ is still $W^{1,p}$ equivalent to a Euclidean ball with $p\to\infty$ as $\epsilon\to 0$ .\\

Other interesting extensions are obtained by looking at Jones' $\beta$-numbers, which are ``one-sided versions'' of the quantity $\theta$ introduced above, for example:
\begin{gather}
	\beta_{\infty}(k;x,r)=\beta_{\infty}(x,r)\equiv r^{-1}\inf_{L^k} {\sup_{y\in \B r x \cap S}} d(y,L) \, .
\end{gather}

Notice in particular that any subset $S$ of a $k$-dimensional plane satisfies $\beta_\infty(x,r)=0$ for all $x,r$, while $\theta(x,r)$ need not be null if there are holes in $S$. These quantities are considered in \cite{davidtoro,NV_harmonic,ENV} to prove $k$-dimensional volume bounds on $S$ and rectifiability. Also, Azzam-Tolsa in \cite{azzam-tolsa} and Tolsa \cite{tolsa:jones-rect} further investigated the summability properties of beta numbers, and showed that a set is $k$-rectifiable if and only if $\int \beta^2(x,r)\frac{dr}{r}<\infty $ for a.e. $x\in S$, and similar statements hold for measures. In \cite{ENV}, \cite{N_RectReif} effective statements of this form were proved and then vastly generalized to classify all measures under analogous $L^2$ Dini estimate assumptions.  These sort of estimates play a role in the analysis of singularities, see for instance \cite{NV_harmonic, NaVaApprox, DeMaSpVa,EdEn, NaVaVarifold, Boyu}. \\

In this note we take a different approach, and hope to drop the Dini estimate in certain contexts.   We show that if a closed set $S\subset \B 2 0 \subseteq \R^m$ satisfies an $\epsilon$-Reifenberg condition with respect to $k$-planes which are uniformly positive with respect to an almost calibration, then $S\cap \B 1 0 $ is $k$-Ahlfors regular and $k$-rectifiable without the additional Dini assumption.  In order to make this precise let us introduce our terminology a little better:

\begin{definition}[$\epsilon$-Calibration]
 Let $\Omega$ be a smooth $k$-form over $B_2(0)\subseteq \R^m$. We say that $\Omega$ is an $\epsilon$-calibration if 
 \begin{enumerate}
  \item $|\Omega-\Omega_0|\leq \epsilon$ for a constant form $\Omega_0$ , 
  \item for all $x\in \R^m$ and any oriented $k$-dimensional subspace $L\subseteq \R^m$, we have $\Omega[L]\leq 1+\epsilon$\, .
 \end{enumerate}
\end{definition}
\begin{remark}
	If $L^k$ is an oriented subspace then we define $\Omega[L]=\Omega[e_1,\ldots,e_k]$ where $e_1,\ldots e_k$ is any oriented orthonormal basis of $L$.
\end{remark}
\begin{remark}
Observe that if $\Omega$ is {\it any } $C^0$ regular $k$-form, then we can normalize it by a constant to satisfy $(2)$, and thus after dilating a small ball $B_s(p)\to B_1(0)$ we have that $\Omega$ becomes an $\epsilon$-calibration.  In particular, the condition on $\Omega$ is highly non-restrictive and holds locally for all $C^0$ $k$-forms.
\end{remark}

Our main result is that if our $\epsilon$-Reifenberg surface is uniformly positive with respect to an almost calibration $\Omega$, then it must be rectifiable with Ahlfor's regularity estimates:

\begin{theorem}[Rectifiable Reifenberg for Almost Calibrations]\label{t:main}
 Let $S\subset B_2(0)\subseteq \R^n$ be a closed set with $0\in S$ and let $\Omega$ be an $\epsilon$-calibration. Then for all $2\epsilon<\alpha<1$ $\exists$ $\delta(n,\epsilon)>0$ and $ A(n,\alpha,\epsilon)>1$ such that if for all $B_{r}(x)\subseteq \B 2 0$ there exists an oriented $k$-dimensional subspace $L=L_{x,r}$ such that
 \begin{align}
 	d_H\Big( S\cap B_r(x), L_{x,r}\cap B_r(x)\Big)<\delta r\, ,\;\;\;\Omega[L_{x,r}]>\alpha>0\, ,
 \end{align}
 then $S\cap \B 1 0$ is $k$-rectifiable and for all $x\in S$ with $B_{2r}(x)\subseteq \B 2 0$ we have that
\begin{gather}
 (1-C(n)\delta)\leq \frac{\cH^k(S\cap \B r x)}{\omega_k r^k}\leq A\, .
\end{gather}
Further, $A\to 1$ as $\alpha\to 1$ and $\epsilon\to 0$ .
\end{theorem}
\begin{remark}
 We underline again that in this context there is no need for a summability condition to ensure rectifiability and volume bounds for $S$, unlike the general Rectifiable Reifenberg case.  We simply require the Reifenberg condition, albeit restricted to positively oriented subspaces.
\end{remark}
\begin{remark}
Though the theorem is stated in Euclidean space, it is easy to see that it can be applied to a manifold locally by looking at a suitably small chart.  Indeed the statements are highly nonsensitive to perturbation.  
\end{remark}
\begin{remark}
	In particular if $\Omega[L_{x,r}]>1-\epsilon$ , so that $S$ is almost calibrated with respect to $\Omega$, then the volume of all balls $B_r(x)$ on $S$ is close to that of the Euclidean ball.
\end{remark}
\begin{remark}
If we view $S$ equipped with the Hausdorff measure as an integral current, oriented by imposing that $\Omega[T_x S] \geq \alpha$, then $\partial S = 0$ as currents in $B_2$ and $S$ satisfies the following ``weak'' almost-minimizing property: if $S'$ is any other integral current with $S' = S$ outside $U \subset\subset B_2$ and $\partial S'\cap U = \emptyset$, then $||S||(U) \leq C(\eps) ||S'||(U)$ .
\end{remark}

The very rough outline is based on the simple observation that one can use the family of Reifenberg approximations $S_r$ as a homotopy to an $\alpha$-positive $k$-plane $L=S_0$.  Up to some care on boundary terms, one can use this to uniformly control the integral of $\Omega$ on $S$, and hence its volume.  We can then take a limit and make a similar conclusion on $S$ itself. \\

\subsection{Examples and Applications}

Most of the immediate applications arise by considering either calibrations or their associated nonintegrable counterparts.  The four examples we will consider are almost complex manifolds, potentially nonintegrable $G2$ and $\text{Spin}(7)$ manifolds, and almost special Lagrangians:

\subsubsection{Almost Complex Manifolds}\label{sss:applications:almost_complex}

Let $(M^{2n},g,J)$ be a manifold with $C^0$ almost complex structure $J$ and compatible metric $g$, i.e. $g(JX,JY) = g(X,Y)$ and $J^2 = -Id$.  We can write $\omega[X,Y] = g(J X,Y)$ to be the associated continuous $2$-form.\\  

In a neighborhood of each point $x\in M$ we can therefore take coordinates $\varphi:B_1(0^{2n})\to M$ with $\varphi(0)=x$ and whose image is $\epsilon$-small, i.e. after normalizing 
\begin{align}\label{e:application:almost_complex}
	|\varphi^*g_{ij} - \delta_{ij}|\, ,\;\;|\varphi^*J - J_0|\, ,\;\; |\varphi^*\omega -\omega_0| < \epsilon\, ,
\end{align}
where the $L^\infty$ norm  is being measured on $B_1(0^{2n})\subseteq \dC^n$ with $J_0$  the standard complex structure on $\dC^n$ and $\omega_0$  the standard symplectic form on $\dC^n$.  In particular, it follows for each $k$ that $\omega^k$ is an $\epsilon$-calibration.  If we apply Theorem \ref{t:main} to this context we arrive at

\begin{corollary}[$C^0$ Almost Complex Manifolds]
	Let $(B_2(0^{2n}), J)$ be an almost complex structure as in \eqref{e:application:almost_complex}, so with $\omega= g(J\cdot,\cdot)$ we have that $\omega^k$ is an $\epsilon$-calibration.  Let $S=S^{2k}\subseteq B_2(0^{2n})$ be $\delta(x,\alpha)$-Reifenberg with respect to $\alpha$-positive subspaces.  That is, assume for each $x\in S$ with $B_{r}(x)\subseteq B_2$ there exists a subspace $L^{2k}_{x,r}\subseteq \dR^{2n}$ with 
\begin{align}
	d_H\big( S\cap B_r(x), {L_{x,r}}\cap B_r(x)\big)<\delta\, ,\;\;\; \omega^k[L]>\alpha>0\, .
\end{align}
Then $S$ is $2k$-rectifiable, and for each $x\in S$ with $B_{2r}(x)\subseteq B_2$ we have the Ahlfors regularity
\begin{align}
	(1-C(n)\delta)\leq \frac{\cH^k(S\cap \B r x)}{\omega_k r^k}\leq A(n,\alpha)\, .
\end{align}
\end{corollary}
\begin{remark}
	If $\alpha\geq 1-\epsilon$ then we might call $S$ almost complex. In particular, if $L_{x,r}$ are almost-complex subspaces, the theorem holds. Notice also that $A\to 1$ as $\epsilon,\delta\to 0$.
\end{remark}

\begin{remark}
If $J = J_0$ is the standard complex structure on $\C^n$ and both $\delta \to 0$ and $\alpha \to 1$ as $r \to 0$ (i.e. $S$ is ``vanishing complex Reifenberg-flat'' ), then $S$ is in fact a complex subvariety of $\C^n$.  This follows by thinking of $S$ as a complex current without boundary and applying the classification result of \cite{King}.
\end{remark}

\vspace{.3cm}

\subsubsection{NonIntegrable $G2$ Structures}

Recall that one defines the associative three form $\psi_0$ and coassociative four form $\Omega_0$ on $\dR^7$ by
\begin{align}
	&\psi_0 = e^{123}- e^{167}- e^{527}- e^{563} -e^{415}- e^{426}- e^{437}\, ,\notag\\
	&\Omega_0 = e^{4567}-e^{4523}-e^{4163}-e^{4127}-e^{2637}-e^{1537}-e^{1526}\, . 
\end{align}
The definitions are motivated by looking at the imaginary part of the octonian algebra.\\

Let $(M^7,\Omega)$ be a $7$-manifold with $C^0$ regular $3$-form $\psi$.  We call $\psi$ a $G2$-structure if at each $x\in M$ there is a basis $e_1,\ldots,e_7\in T_xM$ such that
\begin{align}
	\psi(x)=\psi_0\, .
\end{align}
Recall \cite{Bryant_Holonomy} that such a basis uniquely defines a metric $g$ on $M$, and hence by Hodge starring we get a global continuous $4$-form $\Omega$ such that with respect to the same basis we have $\Omega(x)=\Omega_0$ .\\

In this context we can then find in a neighborhood of each $x\in M$ coordinates $\varphi:B_2(0^{7})\subseteq \dR^7\to M$ with $\varphi(0)=x$ and
\begin{align}\label{e:application:almost_G2}
	\big|\varphi^*\Omega - \Omega_0\big| <\epsilon\, . \\\notag
\end{align}

Our main result is to study Reifenberg sets which are positive with respect to the coassociative form:

\begin{corollary}[$C^0$ Almost G2 Structures]
	Let $(B_2(0^{7}), \Omega)$ be a G2 structure as in \eqref{e:application:almost_G2}.  Let $S^{4}\subseteq B_2(0^{7})$ be $\delta(x,\alpha)$-Reifenberg with respect to $\alpha$-positive subspaces.  That is, assume for each $x\in S$ with $B_{r}(x)\subseteq B_2$ there exists a subspace $L^{4}_{x,r}\subseteq \dR^{7}$ with 
\begin{align}
	d_H\big( S\cap B_r(x), L_{x, r} \cap B_r(x)\big)<\delta\, ,\;\;\; \Omega[L]>\alpha>0\, .
\end{align}
Then $S$ is $4$-rectifiable, and for each $x\in S$ with $B_{2r}(x)\subseteq B_2$ we have the Ahlfors regularity
\begin{align}
	(1-C(n)\delta)\leq \frac{\cH^k(S\cap \B r x)}{\omega_k r^k}\leq A(n,\alpha)\, .
\end{align}
\end{corollary}
\begin{remark}
	If $\alpha\geq 1-\epsilon$ then we might call $S$ almost coassociative, and in this case $A\to 1$ as $\epsilon,\delta\to 0$.
\end{remark}

\vspace{.3cm}

\subsubsection{Almost $Spin(7)$ Manifolds}

Recall that one defines the canonical four form $\Omega_0$ on $\dR^8$ by
\begin{align}
	\Omega_0 &= e^{1256}+e^{1278}+e^{3456}+e^{3478}+e^{1357}-e^{1368}-e^{2457}\, ,\notag\\
	&+e^{2468}-e^{1458}-e^{1467}-e^{2358}-e^{2367}+e^{1234}+e^{5678}\, .
\end{align}
The isomorphism group of this four form is exactly $Spin(7)$ .  We call a manifold $(M^8,\Omega)$ with a $C^0$ four form $\Omega$ an almost $Spin(7)$ manifold if for each $x\in M$ there exists a basis $e_1,\ldots,e_8\in T_xM$ such that
\begin{align}
	\Omega(x) = \Omega_0\, .
\end{align}

In this context we can then find in a neighborhood of each $x\in M$ coordinates $\varphi:B_2(0^{8})\subseteq \dR^8\to M$ with $\varphi(0)=x$ and
\begin{align}\label{e:application:almost_spin7}
	\big|\varphi^*\Omega - \Omega_0\big| <\epsilon\, . \\\notag
\end{align}

Our main result is to study Reifenberg sets which are positive with respect to the $Spin(7)$ form:

\begin{corollary}[$C^0$ Almost $Spin(7)$ Structures]
	Let $(B_2(0^{8}), \Omega)$ be a G2 structure as in \eqref{e:application:almost_spin7}.  Let $S^{4}\subseteq B_2(0^{8})$ be $\delta(x,\alpha)$-Reifenberg with respect to $\alpha$-positive subspaces.  That is, assume for each $x\in S$ with $B_{r}(x)\subseteq B_2$ there exists a subspace $L^{4}_{x,r}\subseteq \dR^{8}$ with 
\begin{align}
	d_H\big( S\cap B_r(x), L_{x, r}\cap B_r(x)\big)<\delta\, ,\;\;\; \Omega[L]>\alpha>0\, .
\end{align}
Then $S$ is $4$-rectifiable, and for each $x\in S$ with $B_{2r}(x)\subseteq B_2$ we have the Ahlfors regularity
\begin{align}
	(1-C(n)\delta)\leq \frac{\cH^k(S\cap \B r x)}{\omega_k r^k}\leq A(n,\alpha)\, .
\end{align}
\end{corollary}

\vspace{.3cm}

\subsubsection{Almost Special Lagrangians}

Let $(M^{2n},\Omega)$ be a $2n$-dimensional manifold with $C^0$ regular $n$-form $\Omega$.  Assume at each $x\in M$ we can find a basis $\frac{\partial}{\partial x^1},\frac{\partial}{\partial y^1},\ldots,\frac{\partial}{\partial x^n},\frac{\partial}{\partial y^n}\in T_xM$ such that
\begin{align}
	\Omega(x) = Re\Big(dz^1\wedge\cdots\wedge dz^n\Big) \, ,
\end{align}
where $dz^j = dx^j+idy^j$ with respect to the dual basis on $T^*_xM$.
Then we call $\Omega$ an almost Special Lagrangian structure.  Note the above is quite weak, and even on a K\"ahler manifold the result would be much weaker than a calibration for Special Lagrangians.  Indeed it does not require a reduction to $SU(n)$ of the structure group, as the required basis may rotate the complex factor in front of $Re\Big(dz^1\wedge\cdots\wedge dz^n\Big)$ as $x$ moves.\\

In this context we can then find in a neighborhood of each $x\in M$ coordinates $\varphi:B_2(0^{2n})\subseteq \dC^n\to M$ with $\varphi(0)=x$ and
\begin{align}\label{e:application:almost_special_lagrangian}
	\big|\varphi^*\Omega - Re\Big(dz^1\wedge\cdots\wedge dz^n\Big)\big| <\epsilon\, . \\\notag
\end{align}

We can state the main result, which is that sets $S^n$ which are Reifenberg flat and uniformly positive with respect to an almost calibration for Special Lagrangians are rectifiable:\\

\begin{corollary}[$C^0$ Almost Special Lagrangians]
	Let $(B_2(0^{2n}), \Omega)$ be an almost calibration for Special Lagrangians as in \eqref{e:application:almost_special_lagrangian}.  Let $S^{n}\subseteq B_2(0^{2n})$ be $\delta(x,\alpha)$-Reifenberg with respect to $\alpha$-positive subspaces.  That is, assume for each $x\in S$ with $B_{r}(x)\subseteq B_2$ there exists a subspace $L^{n}_{x,r}\subseteq \dR^{2n}$ with 
\begin{align}
	d_H\big( S\cap B_r(x), L_{x, r}\cap B_r(x)\big)<\delta\, ,\;\;\; \Omega[L]>\alpha>0\, .
\end{align}
Then $S$ is $n$-rectifiable, and for each $x\in S$ with $B_{2r}(x)\subseteq B_2$ we have the Ahlfors regularity
\begin{align}
	(1-C(n)\delta)\leq \frac{\cH^k(S\cap \B r x)}{\omega_k r^k}\leq A(n,\alpha)\, .
\end{align}
\end{corollary}
\begin{remark}
	If $\alpha\geq 1-\epsilon$ then we might call $S$ an almost Special Lagrangian, and in this case $A\to 1$ as $\epsilon,\delta\to 0$.
\end{remark}

\section{Proof of Theorem \ref{t:main}}

\subsection{The Family of Approximations \texorpdfstring{$S_r$}{Sr}}\label{ss:Sr_construction}

Let $S\subseteq B_2\subseteq \dR^n$ be an $\delta$-Reifenberg set, namely for each $x\in S$ and $B_r(x)\subseteq B_2$ there exists an a $k$-plane $L_{x,r}\subseteq \dR^n$ such that $d_H(S\cap B_r(x),L_{x,r}\cap B_r(x))<\delta r$ .  The classical Reifenberg tells us that near $B_{1}$ we have that $S$ is homeomorphic to a disk.  A key ingredient in this proof is the construction of smooth approximations $S_r$ of $S$, and their careful analysis.  The below is a mild extension of Reifenberg's original construction, which itself follows immediately from his construction:

\begin{lemma}[Construction of Smooth Approximations $S_r$]\label{l:construction_Sr}
	Let $S\subseteq B_{1+\epsilon}(0)\subseteq \dR^n$ be a $\delta$-Reifenberg set as above with $\epsilon>0$ fixed and $\delta<\delta(n,\epsilon)$.  For each $r\in (0,\epsilon]$ there exists smooth oriented manifolds $S_r\subseteq \dR^n$ which depend smoothly on $r$ such that
\begin{enumerate}
	\item $S_\epsilon = L_{0,\epsilon}$ is the fixed plane which well approximates $S$ on $B_{1+\epsilon}(0)$ .
	\item On $\dR^n\setminus B_{1+\epsilon}$ we have that $S_r= S_\epsilon = L_{0,\epsilon}$ is a fixed plane. 
	\item $d_H\big(S\cap B_{1},S_r\cap B_{1}\big)<C(n, \epsilon)\delta\, r$ with second fundamental form bound $r|A_{S_r}|\leq C(n,\epsilon)\delta$.
	\item For each $x\in S_r\cap B_{1}$ and $y\in S\cap B_{r}(x)$ we have $d_{Gr}\big(L_{y,r_y}, T_x S_r\big)<C(n,\epsilon)\delta$, where $r_y=r_{|y|}$ is piecewise linear with $r_{|y|}=r$ for $|y|\leq 1$ and $r_{|y|}=\epsilon$ for $|y|\geq 1+\epsilon$ .  In particular we can orient $S_r$ so that $\Omega_0[T_x S_r] > \eps/2$ for all $x \in S_r \cap B_1$.
	\item $\big|\frac{d}{dr}S_r\big| < C(n,\epsilon)\delta$ .
	\item If $S\cap B_{2r}(x)$ is smooth with second fundamental form $r|A_S|\leq \epsilon$, then $S\cap B_{r}(x)=S_r\cap B_r(x)$ . 
\end{enumerate}
\end{lemma}
\begin{remark}
	$d_{Gr}$ in $(4)$ is the Grassmann distance between the subspace $L_{y,r_y}$ and the tangent space $T_x S_r$ .
\end{remark}
\begin{remark}
	$\frac{d}{dr}S_r$ is understood as the normal velocity vector to $S_r$, so that $(5)$ is the statement that the this normal velocity vector has small norm.
\end{remark}
\begin{remark}
	Condition $(4)$ is just saying that $T_xS_r$ is always well approximated by a subspace $L$, which itself well approximates $S$ on the appropriate scale.  On $B_1$ this scale is $r$ and outside of $B_{1+\epsilon}$ this scale is $\epsilon$, with a linear interpolation between.
\end{remark}

\begin{proof}[Rough Proof]
	As the details are in many references, see for instance \cite{toro:reifenberg,simon_reif,N_RectReif},  and in particular \cite[Theorem 4.2]{N_RectReif} for the construction of a one parameter family of approximating submanifolds $S_r$, let us just mention a few words on the proof.  The idea is simple if somewhat tedious.  One picks a Vitali covering $\{B_{r_i}(y_i)\}$ of $S\cap B_{1+\epsilon}$ with $r_i=r_{|y_i|}$ as in $(4)$ .  Vitali here means that $\{B_{r_i/5}(y_i)\}$ are disjoint.  Then choose for each such ball the best approximating plane $L_i=L_{y_i,r_i}$ , and then glue them together using estimates provided by the Reifenberg assumption.  Condition $(5)$ is not usually stated in this form, instead one often points out that $S_{2r}$ is graphical over $S_r$ with likewise bounds.  However, the graphical property immediately implies we can construct as in $(5)$ .  Indeed, one can build $S_a=S_{r_a}$ for $r_a=2^{-a}$ and use the graphical property to interpolate and build $S_r$ for $r\in (2^{-a},2^{-a+1})$ .
\end{proof}

Let us state in words the outcome of Lemma \ref{l:construction_Sr}.  Beginning with our Reifenberg set $S$ we are building a family of approximations $S_r$, which by $(1)$ begin at a fixed plane $S_\epsilon = L_{0,\epsilon}$ and end at $S_0 = S$, at least on $B_{1}$ .  This family $S_r$ is independent of $r$ outside of $B_{1+\epsilon}$, and hence equal to the same plane $L_{0,\epsilon}$.  By $(3)$ each $S_r$ is $\delta r$-close to $S$ on $B_{1}$ .  The family $S_r$ is moving in a continuous fashion by $(5)$, and the tangent space $T_x S_r$ is always close to some approximating plane of $S$ by $(4)$ .  Condition $(6)$ tells us that if $S$ is itself scale invariantly smooth on some scale, then the approximation should equation $S$ itself.

\subsection{Basic volume bounds on \texorpdfstring{$S$}{S} and \texorpdfstring{$S_r$}{Sj}}

Let us begin with the following, which is under the context of Theorem \ref{t:main} and Lemma \ref{l:construction_Sr}:

\begin{lemma}\label{l:Sr_B1_vol}
	Let $S\subseteq B_{1+\epsilon}$ be as in Theorem \ref{t:main}.  Then for $A=A(n,\alpha,\epsilon)$, $\epsilon<\epsilon(n,\alpha)$ and $\delta<\delta(n,\alpha,\epsilon)$ we have the volume estimate
\begin{align}
	 \cH^k(S_r\cap B_1)&\leq \cH^k(S_r\cap B_{1+\epsilon})\leq A\omega_k\, ,\notag\\
	 \cH^k(S_r\cap B_{1})&\geq (1-C(n)\delta)\omega_k
\end{align}
independent of $r>0$ with $A(n,\alpha,\epsilon)\to 1$ as $\alpha\to 1$ and $\epsilon\to 0$ . Moreover, as currents we have
\[
(\pi_0)_\sharp [S_r] = [\pi_0(S_r)], \quad \text{ and } \quad \pi_0(S_r) \supset B_{1-C(n)\delta} \cap L_{0, \eps},
\]
where $\pi_0 : \R^n \to L_{0, \eps}$ is the orthogonal projection map.

\end{lemma}
\begin{proof}
	Recall that our almost calibration $\Omega$ satisfies the $L^\infty$ control
\begin{align}
	|\Omega[L]|\leq 1+\epsilon\, ,|\Omega-\Omega_0|<\epsilon\, ,
\end{align}
where $\Omega_0$ is a constant $k$-form.  It follows from Lemma \ref{l:construction_Sr}.4 and our uniform positivity that for each $x\in S_r$ we have
\begin{align}\label{e:Volume_Sr:1}
	\Omega_0[T_x S_r]\geq \Omega[T_x S_r]-\epsilon\geq \Omega[L_{y,r_y}]-C(n,\epsilon)\delta-\epsilon\geq \alpha-3\epsilon/2\, ,
\end{align}
where by Lemma \ref{l:construction_Sr}.3 $y\in S \cap B_r(x)$ is any choice of point.

On the other hand, $d\Omega_0=0$ is a closed $k$-form trivially as it is constant, and so using Lemma \ref{l:construction_Sr}.1 and Lemma \ref{l:construction_Sr}.5 we have that
\begin{align}
	\int_{S_r\cap B_{1+\epsilon}} \Omega_0 &= \int_{L_{0,\epsilon}\cap B_{1+\epsilon}} \Omega_0  \leq 1+C(n)\epsilon\, .
\end{align}
If we then combine this with \eqref{e:Volume_Sr:1} we can turn this into the upper bound
\begin{align}
	\cH^k(S_r\cap B_1) &= \int_{S_r\cap B_1} d\cH^k_S \leq \int_{S_r\cap B_{1+\epsilon}} d\cH^k_S \, ,\notag\\
	&\leq (\alpha-3\epsilon/2)^{-1}\omega_k\int_{S_r\cap B_{1+\epsilon}} \Omega_0 \leq \frac{1+C(n)\epsilon}{\alpha-3\epsilon/2}\omega_k\, .
\end{align}

In order to prove the lower bound, we proceed in a manner similar to the general Reifenberg case.  Namely consider the projection map $\pi_0:\dR^n\to L_{0,\epsilon}$, which is clearly a Lipschitz map with $|d\pi_0[v]|\leq |v|$ .  Note that
\begin{align}
	\pi_0[S_r\cap \partial B_{1+\epsilon}(0^n)] = L_{0,\epsilon}\cap \partial B_{1+\epsilon}(0^k)\, ,
\end{align} 
independent of $r>0$ , and that $\pi_0$ on $S_\epsilon\cap B_{1+\epsilon}(0^n)$ is effectively just the identity map.  In particular, for all $r>0$ we have that $\pi_0$ on $S_r\cap B_{1+\epsilon}(0^n)$ has topological degree one.  Hence $\pi_0$ must map $S_r\cap B_{1+\epsilon}(0^n)$ onto $B_{1+\epsilon}(0^k)$.  On the other hand, it follows from the Hausdorff condition that we then must have
\begin{align}
	\pi_0[S_r\cap B_{1}(0^n)]\supseteq B_{1-C(n)\delta}(0^k)\, .
\end{align}
As $\pi_0$ is a submetry, we then get the lower volume estimate
\begin{align}
	\cH^k(S_r\cap B_{1})\geq \cH^k(\pi_0[S_r\cap B_{1}])\geq \cH^k(B_{1-C(n)\delta}(0^k))\geq (1-C(n)\delta)\omega_k\, ,
\end{align}
as claimed.  The assertion $(\pi_0)_\sharp [S_r] = [\pi_0(S_r)]$ follows from the constancy theorem and the fact that $\pi_0$ has degree one on $S_r \cap B_{1+\eps}$. 
\end{proof}

Let us now use the above to conclude Ahlfor's regularity of our approximating submanifolds:\\

\begin{lemma}\label{l:Sr_Br_vol}
	Let $S\subseteq B_2$ be as in Theorem \ref{t:main}, and let $S_r$ be a smooth approximation with $S_r\to S$ on $B_1$ as in Lemma \ref{l:construction_Sr}.  Let $A=A(n,\alpha,\epsilon)$, $\epsilon<\epsilon(n,\alpha)$ and $\delta<\delta(n,\alpha,\epsilon)$, then for all $x\in S_r$ with $B_{2s}(x)\subseteq B_{2}$ we have the volume estimates
\begin{align}
	 (1-C(n)\delta)\omega_k\,s^k \leq \cH^k(S_r\cap B_s(x))&\leq A\omega_k\,s^k\, .
\end{align}
Here $A(n,\alpha,\epsilon)\to 1$ as $\alpha\to 1$ and $\epsilon\to 0$ .  Moreover, when $r < s$ then given any $y \in S \cap B_r(x)$ we have the identity
\begin{equation}\label{eqn:Sr_Br_vol2}
(\pi_{y, s})_\sharp [S_r \cap B_s(x)] = [\pi_{y, s}(S_r \cap B_s(x)] \quad \text{ and } \quad \pi_{y, s}(S_r \cap B_s(x)) \supset L_{y, s} \cap B_{(1-C(n)\delta)s}(x).
\end{equation}
where $\pi_{y, s}$ is the projection onto $L_{y, s}$.

\end{lemma}
\begin{proof}
	Take $x\in S_r$ with $B_{2s}(x)\subseteq B_{2}$.  Note that if $s\leq r$ then the result follows due to the smoothness conditions on $S_r$, indeed one can take $A\approx 1+C(n)\delta$ using the second fundamental form bounds by Lemma \ref{l:construction_Sr}.3 . So we can assume $s>r$ without loss.  If we rescale and translate our ball $B_{2s}(x)\to B_{2}(0)$ with $\tilde S_r\subseteq B_2(0)$ the rescaled set, observe that $\tilde S_r$ is now itself a Reifenberg set which satisfies the context of Lemma \ref{l:construction_Sr} and Lemma \ref{l:Sr_B1_vol}.  Further the approximations $\tilde S_{r,r'}$ of $\tilde S_{r}$ from Lemma \ref{l:construction_Sr} satisfy $\tilde S_{r,r}=\tilde S_r$ on $B_1$, as in Lemma \ref{l:construction_Sr}.6.  Hence we can apply Lemma \ref{l:Sr_B1_vol} to $\tilde S_{r,r}$ to get the claimed estimates
	\begin{align}
		\cH^k(S_r\cap B_s(x))&=s^k \cH^k(\tilde S_r\cap B_1(0))\leq A\omega_k\,s^k\, ,\notag\\
		\cH^k(S_r\cap B_{s}(x))&=s^k \cH^k(\tilde S_r\cap B_1(0))\geq (1-C(n)\delta)\omega_k\,s^k\, .
	\end{align}
The last statement follows in a similar fashion from Lemmas \ref{l:construction_Sr} and \ref{l:Sr_B1_vol}.
\end{proof}

\subsection{Rectifiability and Volume Control of \texorpdfstring{$S$}{S}}

Consider smooth approximations $S_r$ of $S$ as in Lemma \ref{l:construction_Sr} such that
\begin{align}
	S_r = L_{0,\epsilon} \text{ outside of }B_{1+\epsilon}\, ,\notag\\
	d_H(S_r\cap B_{1},S\cap B_{1})<C(n)\delta r\, .
\end{align}

The $S_r$ are oriented $k$-dimensional surfaces in $B_{1+\eps}$ with $\partial S_r = 0$ and uniform bounds $\haus^k(S_r \cap B_{1+\eps}) \leq A \omega_k$.  Given any $r_i \to 0$, we can therefore apply the Feder-Fleming compactness theorem for currents \cite{simon:gmt} to pass to a subsequence and find an integral $k$-current $T$ in $B_{1+\eps}$ so that $[S_{r_i}] \to T$.  The limiting current $T$ has zero boundary and is positively oriented w.r.t. $\Omega_0$ (since the $S_{r_i}$ are).

We claim that $S \cap B_1$ is $k$-rectifiable, and $T = [S]$ where we orient the tangent spaces of $S$ positively w.r.t. $\Omega_0$, i.e. $\Omega_0[T_x S] \geq 0$ for $\haus^k$-a.e. $x$.  (Our claim will also imply $[S_r] \to [S]$ in $B_1$ as $r \to 0$.)

Trivially $\spt (T) \cap B_1 \subset S$, and so in $B_1$ we can write the mass measure $||T|| = \haus^k \llcorner \theta \llcorner S'$ for some $S' \subset S$ and some $\N$-valued function $\theta$.  On the other hand, for any $x \in S \cap B_1$ and $B_{2s}(x) \subset B_{2}$ we can take the limit of \eqref{eqn:Sr_Br_vol2} to get
\begin{equation}\label{eqn:rect-1}
(\pi_{x, s})_\sharp (T \llcorner B_s(x)) = [U_{x, s}] \quad \text{ for some } L_{x, s} \cap B_{(1-C(n)\delta) s}(x) \subset U_{x, s} \subset L_{x, s}.
\end{equation}
Since $\Lip(\pi_{x, s}) \leq 1$, we get
\begin{equation}\label{eqn:rect-2}
||T||(B_s(x)) \geq (1-C(n)\delta) s^k \omega_k \quad \forall s < 1-|x|/2,
\end{equation}
and therefore we deduce $S' = S$.  

So $S$ is $k$-rectifiable, and for $\haus^k$-a.e. $x \in S \cap B_1 \equiv \spt T \cap B_1$, $T$ has the approximate tangent plane-with-multiplicity: 
\[
(\eta_{x,r})_\sharp T \to \theta(x) [T_x S] \quad \text{ as $r \to 0$},
\]
Here $\eta_{x, r}(y) = (y-x)/r$, and $T_x S$ is oriented so that $\Omega_0[T_x S] \geq 0$.  In fact since $\Omega_0[T_y S_r] \geq \eps/2$ for all $y \in S_r \cap B_1$ we have $\Omega_0[T_x S] \geq \eps/2$ also.  The identity \eqref{eqn:rect-1} passes to the limit, so we get
\[
\theta(x) (\pi_{x, s})_\sharp [T_x S] = [L_{x, s}] \quad \forall s < 1-|x|/2,
\]
which implies $\theta(x) = 1$.  This proves our claim. 

From $T = [S]$, \eqref{eqn:rect-2} and lower-semi-continuity of mass we get the volume estimates
\[
(1-C(n)\delta)\omega_k \leq \haus^k(S \cap B_1) \leq A \omega_k.
\]

We can of course now apply the same argument to any ball $B_{2r}(x)\subseteq S$ in order to obtain that $S$ is everywhere rectifiable with the estimates	
\begin{align}
		(1-C(n)\delta)\omega_k\,r^k\leq \cH^k(S\cap B_r(x))&\leq A\omega_k\, r^k\, ,
\end{align}
finishing the proof of Theorem \ref{t:main}.


\bibliographystyle{aomalpha}
\bibliography{ENV_Reifenberg}

\end{document}